\documentclass[10pt]{article}
\usepackage[T1]{fontenc}
\usepackage{lmodern}
\usepackage[a4paper,margin=27mm]{geometry}
\usepackage{amsmath,amssymb,amsthm,mathtools}
\usepackage{microtype}
\usepackage{enumitem}
\usepackage{xcolor}
\usepackage[colorlinks=true,linkcolor=blue!45!black,citecolor=blue!45!black,urlcolor=blue!45!black]{hyperref}

\newtheorem{theorem}{Theorem}
\newtheorem{lemma}{Lemma}

\newtheorem{proposition}{Proposition}
\theoremstyle{remark}
\newtheorem{remark}{Remark}

\hypersetup{pdftitle={Real-analytic standard families and rotational invariant curves}}
\allowdisplaybreaks[1]
\theoremstyle{definition}

\newcommand{\Si}{\operatorname{Si}}

\begin{document}

\title{Real-analytic standard families and rotational invariant curves}
\author{Salvador Addas-Zanata\thanks{The author thanks FAPESP (project 2023/07076-4) for partial financial support.}}
\date{}

\maketitle

\centerline { {\sl Instituto de Matem\'atica e Estat\'\i stica }}

\centerline {{\sl Universidade de S\~ao Paulo}}

\centerline {{\sl Rua do Mat\~ao 1010, Cidade Universit\'aria,}} 

\centerline {{\sl 05508-090 S\~ao Paulo, SP, Brazil}}

\maketitle

\begin{abstract}
In this paper we study one-parameter families of exact standard maps of the form $$\widehat{f}_{k,g}(\widehat{x},\widehat{y})
=\bigl(\widehat{x}+\widehat{y}+kg(\widehat{x}) \bmod 1,\widehat{y}+kg(\widehat{x})\bigr),$$ where $k$ is a real parameter
and $g:\mathbb{T} \rightarrow \mathbb{R}$ is real-analytic, 1-periodic and has zero mean. The most famous example is when 
$g(\widehat{x})=\frac{\sin(2\pi\widehat{x})}{2\pi}$. In this case, the closed subset $A_{\sin(2\pi\cdot)/(2\pi)}\subset \mathbb{R}$ of values of $k$ for which $\widehat{f}_{k,\sin(2\pi\cdot)/(2\pi)}$ has rotational 
invariant curves is conjectured to be a closed interval. Much has been done towards this conjecture, both numerically and analytically. But a complete proof 
has not yet been achieved.
Our approach here goes in the opposite direction: we construct a function $\theta:\mathbb{T} \rightarrow \mathbb{R}$ satisfying the
above conditions, such that $A_{\theta}$ is not an interval: It contains at least 2 different connected components, which are non-degenerate closed intervals.    
\end{abstract}

\vspace{-2em}

\vskip1truecm
\section{Introduction}

The sinusoidal standard family is a classical testing ground for the
existence and destruction of rotational invariant curves (RICs). Its planar lift is
\[
\widetilde f_{k,\sin(2\pi\cdot)/(2\pi)}(\widetilde x,\widetilde y)
=\left(\widetilde x+\widetilde y+\frac{k}{2\pi}\sin(2\pi\widetilde x),
       \widetilde y+\frac{k}{2\pi}\sin(2\pi\widetilde x)\right).
\]
It is not difficult to see that the closed set
$A_{\sin(2\pi\cdot)/(2\pi)}$ of parameters $k$ admitting a RIC is symmetric
about $0$. These essential invariant circles form barriers to vertical
transport, making the disappearance of the last RIC a central question.

At $k=0$ the map is integrable and every horizontal circle is invariant.
Moser's invariant curve theorem \cite{Moser1962}, and subsequent
developments of KAM theory, imply persistence of circles with suitable
Diophantine rotation numbers for all sufficiently small $|k|$.
In particular, $A_{\sin(2\pi\cdot)/(2\pi)}$ contains an interval around $0$.
Obtaining useful explicit sizes for such an interval requires much
sharper estimates than this qualitative statement. A rigorous
computer-assisted result of Celletti and Chierchia
\cite{CellettiChierchia1991} establishes the existence of an analytic
invariant circle with golden rotation number
$\omega_{\mathrm g}=(\sqrt5-1)/2$ for every $|k|\le 0.838$.
Their result gives
\[
[-0.838,0.838]
\subset A_{\sin(2\pi\cdot)/(2\pi)}
\]

The numerical picture is considerably more precise. Greene
\cite{Greene1979} introduced a criterion based on the residues of periodic
orbits whose rotation numbers approximate that of an invariant circle.
His computations suggest that the golden circle is the last to break,
at a critical parameter $k_{\mathrm c}\approx0.971635\ldots$.
MacKay's renormalization approach \cite{MacKay1983} provided a framework
for understanding the scaling and universality observed near breakup.
The resulting conjectural picture is that
$A_{\sin(2\pi\cdot)/(2\pi)}=[-k_{\mathrm c},k_{\mathrm c}]$:
in particular, RICs would never reappear as the positive parameter
increases beyond $k_{\mathrm c}.$

On the nonexistence side, Mather \cite{Mather1984} proved that the
standard map has no RICs for $|k|>4/3$. MacKay and Percival
\cite{MacKayPercival1985} developed rigorous converse-KAM methods and
proved nonexistence for every $|k|\ge 63/64=0.984375$.
Consequently,
\[
A_{\sin(2\pi\cdot)/(2\pi)}
\subset\left(-\frac{63}{64},\frac{63}{64}\right).
\]
Jungreis \cite{Jungreis1991} subsequently proved nonexistence at several
specified parameter values, including $k=0.9718$, much closer to the
numerically predicted transition.

Rigorous existence estimates also approach this transition.
Figueras, Haro and Luque \cite{FHL} validated an invariant circle with
golden rotation number at $k=0.9716$ using an a posteriori KAM theorem.
It is important to distinguish these assertions: existence at a
specified parameter does not by itself establish existence throughout
the interval between that parameter and $0$, and nonexistence at
$k=0.9718$ does not establish nonexistence at every larger parameter.
Thus the proximity of these existence and nonexistence certificates
does not, by itself, prove that $A_{\sin(2\pi\cdot)/(2\pi)}$ is an interval.

This distinction motivates the present work. We replace the sine by a
real-analytic, mean-zero periodic function $g$ and study the analogous
set $A_g$. Starting from Bullett's piecewise-linear example
\cite{Bullett}, we construct real-analytic standard families for which RICs persist
throughout a neighborhood of $k=4/3$, although there are no RICs on an
intervening interval of positive parameters. KAM theory also supplies
RICs throughout a neighborhood of $k=0$. The resulting set $A_g$
therefore has distinct connected components containing $0$ and $4/3$
in their interiors. This shows that the interval structure suggested
for the sinusoidal family cannot be deduced from analyticity and the
twist condition alone. There must be some mysterious property 
hidden with the sine function. 

Related numerical evidence was obtained by Wilbrink \cite{Wilbrink1987},
who studied analytic standard-like families approaching a piecewise-linear
limit. Using Greene's residue criterion, he observed disappearance and
reappearance of invariant circles with prescribed irrational rotation
numbers as the nonlinearity parameter increased. He explicitly noted that
this criterion was being used without a general proof of its validity
\cite[Section~2, p.~360]{Wilbrink1987}; the analytical calculations in
Section~5 of his paper concern the piecewise-linear limit. 

The present work gives a
rigorous construction of this type of reappearance: for a fixed real-analytic
standard family of maps, an interval in the parameter $k$ admitting no RIC of any rotation number separates two
intervals on which RICs exist.

\subsection{Notation, definitions and statement of the main theorem}

\begin{enumerate}
\item  Denote the circle as $\mathbb T=\mathbb R/\mathbb Z$ and the 2-torus
as $\mathbb T^{2}=\mathbb R^{2}/\mathbb Z^{2}.$ Points of the torus $\mathbb %
T^{2}$ are denoted as $z=(x,y),$ points of the plane as $\widetilde{z}=(%
\widetilde{x},\widetilde{y})\in \mathbb R^{2}$, and points of the vertical
cylinder as $\widehat{z}=(\widehat{x},\widehat{y})\in \mathbb T\times %
\mathbb R$. The corresponding covering projections, $\pi :\mathbb %
R^{2}\rightarrow \mathbb T\times \mathbb R,$ $\tau :$ $\mathbb T\times %
\mathbb R\rightarrow \mathbb T^{2}$ and $p:\mathbb R^{2}\rightarrow \mathbb %
T^{2}$ are given by 
\[
\pi (\widetilde{x},\widetilde{y})=(\widetilde{x}\bmod1,\widetilde{y}),\text{ 
}\tau (\widehat{x},\widehat{y})=(\widehat{x},\widehat{y}\bmod1)\text{ and }p(%
\widetilde{x},\widetilde{y})=(\widetilde{x}\bmod1,\widetilde{y}\bmod1)\text{ 
}.
\]

\item  A torus map is written $f$, its specified plane lift $\widetilde{f}$,
and its specified cylinder lift $\widehat{f}$. 

\item Let $p_{1},p_{2}:\mathbb{R}^2 \to\mathbb{R}$ denote the horizontal and vertical
projections in the plane, $p_{1}(\widetilde{x},\widetilde{y})=\widetilde{x}$ and $p_{2}(\widetilde{x},\widetilde{y})=\widetilde{y}$.
The analogous projections in the cylinder and torus are denoted in the same way.

\item  We also use tildes for real lifts of circle maps, such as $\widetilde{%
h}:\mathbb R\to \mathbb R$ lifting $h:\mathbb T\to \mathbb T$. Periodic
scalar functions retain their undecorated names; the same function may be
evaluated at a real argument or its class in $\mathbb T$. 

\item  For a continuous, $1$-periodic function $g:\mathbb R\to \mathbb R$ of
zero mean, consider the one-parameter family of homeomorphisms of the plane  
\begin{equation}
\widetilde{f}_{k,g}(\widetilde{x},\widetilde{y})=\bigl(\widetilde{x}+%
\widetilde{y}+kg(\widetilde{x}),\,\widetilde{y}+kg(\widetilde{x})\bigr)%
,\qquad k\in \mathbb R.  \label{eq:family}
\end{equation}
It induces maps of the cylinder $\widehat{f}_{k,g},$ and of the torus $%
f_{k,g}$, related by the projections above. This is the so-called standard
family of homeomorphisms associated to $g.$ A simple computation shows that $%
\widetilde{f}_{k,g}$ is an area-preserving positive twist map. The twist
property is trivial to see, because
\[
\frac{\partial (\widetilde{x}+\widetilde{y}+kg(\widetilde{x}))}{\partial 
\widetilde{y}}=1,\text{ at all }(\widetilde{x},\widetilde{y})\in \mathbb R^{2}.
\]
and the area-preservation comes from the fact that, defining $\widetilde{V}(%
\widetilde{x},\widetilde{y})=(\widetilde{x},\widetilde{y}+kg(\widetilde{x}))$
and $\widetilde{H}(\widetilde{x},\widetilde{y})=(\widetilde{x}+\widetilde{y},%
\widetilde{y}),$ the map $\widetilde{f}_{k,g}(\widetilde{x},\widetilde{y})=%
\widetilde{H}\circ \widetilde{V}(\widetilde{x},\widetilde{y}).$ And both $%
\widetilde{V}$ and $\widetilde{H}$ preserve area. The $\det (D\widetilde{H}%
\mid _{(\widetilde{x},\widetilde{y})})\equiv 1$ at all $(\widetilde{x},%
\widetilde{y})\in \mathbb R^{2}$ and $\widetilde{V}$ preserves area because it
maps each vertical line to itself by a translation, so Fubini's theorem applies. 

\item  We say that some $\widehat{f}_{k,g}:\mathbb T\times \mathbb %
R\rightarrow \mathbb T\times \mathbb R$ has a \emph{rotational invariant
curve (RIC)} $\gamma,$ if $\gamma $ is an $\widehat{f}_{k,g}$-invariant essential
simple closed curve in the cylinder. The Birkhoff Invariant Curve Theorem 
\cite{fathi} implies that $\gamma $ coincides with the graph of a continuous
function $m:\mathbb T\to\mathbb R$ (Lipschitz when $g$ is Lipschitz).

\item \label{item:parameter-set} Denote by
\[
A_{g}:=\{k\in \mathbb R:\widehat{f}_{k,g}\text{ has a rotational invariant
curve}\}.
\]
For this family, Lemma~\ref{noinvcurves} implies that $A_g$ is closed. And if 
$g\not\equiv0$, it is also bounded. Indeed, if $|k|\|g\|_\infty\ge1$,
continuity and the zero-mean condition give some $\widehat{x}_0\in\mathbb T$
such that for all integers $n,$ %and $a\varepsilon\in\{-1,1\}$ such that $kg(\widehat{x}_0)=\varepsilon$.
%Consequently,
\[
\widehat{f}_{k,g}^{\,n}(\widehat{x}_0,0)
=(\widehat{x}_0,n), \text{ or } \widehat{f}_{k,g}^{\,n}(\widehat{x}_0,0)=(\widehat{x}_0,-n).
\]
Any of the possibilities above clearly rules out a RIC. %: if one existed, its integer vertical
%translates would also be invariant, and the closed strips between successive
%translates would trap every orbit. In particular, for $k\ge0$, the condition
%$k\max g\ge1$ rules out RICs. 
Thus,
\begin{equation}
A_g\subset\left(-\frac1{\|g\|_\infty},\frac1{\|g\|_\infty}\right)
\qquad\text{if }g\not\equiv0.
\label{eq:parameter-bound}
\end{equation}
%Thus $A_g$ is compact whenever $g\not\equiv0$.

\item  For a 1-periodic scalar function $u:\mathbb R\to \mathbb R$, we
distinguish the lifted graph $\widetilde{\Gamma }_{u}=\{(\widetilde{x},u(%
\widetilde{x})):\widetilde{x}\in \mathbb R\}$ from the closed curve $%
\widehat{\Gamma }_{u}=\pi (\widetilde{\Gamma }_{u})$ in the cylinder.
Invariance of either is equivalent to invariance of the other for the
specified maps in the above family. All rotation numbers below refer to
specified degree-one lifts and are real numbers, not merely classes modulo $1
$.
\end{enumerate}

Now we are ready to state our result:

\begin{theorem}
\label{thm:main} There exists a real-analytic, $1$-periodic, mean-zero
function $\theta :$ $\mathbb R\rightarrow \mathbb R$ such that $A_{\theta }$
is not connected. A connected component of $A_{\theta }$ is a closed
interval $I_{0}$ containing $0$ in its interior and another connected
component is a closed interval $I_{4/3}$ containing $4/3$ in its interior.
\end{theorem}

This paper is organized as follows. In Section~\ref{sec:preliminaries} some auxiliary results are proved,
including the robustness of nonexistence of RICs, the Fourier estimates and Herman's construction. 
We also discuss some aspects of Bullett's example in detail.
And in Section~\ref{sec:proof}, the main theorem is proved.

\section{Preliminaries}\label{sec:preliminaries}

In this section we recall some results we use, and prove some auxiliary ones.

\subsection{On the non-existence of RICs}

In this subsection, we present some results that will be used in the proof of part 3
of the main theorem. Let $k \in \mathbb{R}$ and $g$ be a continuous 1-periodic function of the real
line with zero mean.

\begin{lemma}
\label{birkhoff} If $\widehat{f}_{k,g}$ has no RIC, then for all $M>0,$
there exists $\widehat{z}^{-},\widehat{z}^{+}\in \mathbb{T}\times [0,1[$
and integers $n^{-},n^{+}>0$ such that 
\[
p_{2}\circ \widehat{f}_{k,g}^{n^{+}}(\widehat{z}^{+})-p_{2}(\widehat{z}%
^{+})>M\text{ and }p_{2}\circ \widehat{f}_{k,g}^{n^{-}}(\widehat{z}%
^{-})-p_{2}(\widehat{z}^{-})<-M.
\]
\end{lemma}

\begin{proof} As the proof for both cases is analogous, let us consider only the
existence of $\widehat{z}^{+}.$ By contradiction, if for some $M>0,$ there
exists no such point, then 
\[
\widehat{f}_{k,g}^{n}(\mathbb{T}\times ]-\infty ,1[)\subset \mathbb{T}\times ]-\infty ,M+1]\text{ for all }n\geq 0.
\]
This means that the open connected set 
\[
O:=\bigcup _{n=0}^{+\infty }\widehat{f}_{k,g}^{n}(\mathbb{T}\times ]-\infty
,1[)\subset \mathbb{T}\times ]-\infty ,M+1].
\]
Consider $Filled(O),$ which is given by the union of $O$ with all
bounded connected components of its complement. As $\mathbb{T}\times
]-\infty ,1[\subset O$ and $O\cap \mathbb{T}\times ]M+1,\infty [=\emptyset ,
$ we get that $Filled(O)$ is homeomorphic to the cylinder and $\widehat{f}%
_{k,g}(Filled(O))\subset Filled(O).$ As $g$ has zero mean, although $%
\widehat{f}_{k,g}^{-1}(Filled(O))\supset Filled(O),$ they must have the same
closure, see the proof of Corollary~2 of \cite{eubrafa}. This means that 
\[
\overline{\bigcup _{n=0}^{+\infty }\widehat{f}%
_{k,g}^{-n}(Filled(O))}=\overline{Filled(O)}.
\]

So, the set 
\[
O^{*}:=Filled\left[Interior\left( \overline{\bigcup _{n=0}^{+\infty }\widehat{f}%
_{k,g}^{-n}(Filled(O))}\right)\right] 
\]
is $\widehat{f}$-invariant, open, connected, homeomorphic to the cylinder
and $$\mathbb{T}\times ]-\infty ,1[\subset O^{*}\subset \mathbb{T}\times
]-\infty ,M+2[.$$ Thus, by the Birkhoff invariant curve theorem, its boundary
must be a RIC, a contradiction with our hypothesis.
\end{proof}

Then we have the following generalization of a lemma from \cite{non1}.

\begin{lemma}
\label{lemBD01} Suppose that for points $\widetilde{z}^{-},\widetilde{z}%
^{+}\in [0,1]^{2}$ and integers $n^{-},n^{+}>0,$ the following holds:
\[
p_{2}\circ \widetilde{f}_{k,g}^{n^{+}}(\widetilde{z}^{+})-p_{2}(\widetilde{z}%
^{+})>8+4\left| k\right| M\text{ and }p_{2}\circ \widetilde{f}_{k,g}^{n^{-}}(%
\widetilde{z}^{-})-p_{2}(\widetilde{z}^{-})<-8-4\left| k\right| M,
\]
where $k$ is a fixed real number, $g:$ $\mathbb{R}\rightarrow \mathbb{R}$ is a 1-periodic continuous function 
with zero mean and $M=\left\| g\right\| _{\infty }.$   Then, there exist points $\widetilde{z},\widetilde{w}\in
[0,1]^{2}$ and integers $n_{1}>0$ and $n_{2}>0$ such that

\begin{itemize}
\item[$(1)$]  $\widetilde{f}_{k,g}^{n_{1}}(\widetilde{z})=\widetilde{z}+(s,%
\mathrm{pos})$, for some $s\in \mathbb{Z}$ and $\mathrm{pos}\geq 2,$

\item[$(2)$]  $\widetilde{f}_{k,g}^{n_{2}}(\widetilde{w})=\widetilde{w}+(l,\mathrm{neg})$, for some $l\in \mathbb{Z}$ and $\mathrm{neg}%
\leq -2$.
\end{itemize}
\end{lemma}

\begin{proof} The proof of the existence of $\widetilde{w}$ is analogous to the one
for $\widetilde{z},$ so we omit it. 

Initially, assume that for all $\widetilde{z}\in [0,1]^{2}$ such that 
$\widetilde{f}_{k,g}^{n^{+}}(\widetilde{z})=\widetilde{z}+(s,m)$ for some
integer $s,$ we have $m<2.$

A very natural thing is to consider the linear segment $r=\{p_{1}(\widetilde{%
z}^{+})\}\times [0,1].$ First, let us define 
\begin{equation}
Max.H.L(\widetilde{f}_{k,g}^{n}(r))={\sup_{t,t^{\prime }\in [0,1]}}\left|
p_{1}\circ \widetilde{f}_{k,g}^{n}(p_{1}(\widetilde{z}^{+}),t)-p_{1}\circ \widetilde{f}_{k,g}^{n}(p_{1}(\widetilde{z}^{+}),t^{\prime
})\right| .  \label{chma}
\end{equation}

It is clear that 
\begin{equation}
Max.H.L(\widetilde{f}_{k,g}^{n}(r))\geq n.  \label{chmaxde}
\end{equation}

So for all $n\geq 1,$ there exists at least one integer $s,$ such that $%
p_{1}(\widetilde{z}^{+})+s\in p_{1}\left( \widetilde{f}_{k,g}^{n}(r)\right) .
$

The hypothesis we assumed in the beginning implies that for all $0\leq t\leq 1$
such that

\begin{equation}
p_{1}\circ \widetilde{f}_{k,g}^{n^{+}}(p_{1}(\widetilde{z}^{+}),t)=p_{1}(%
\widetilde{z}^{+})\text{ }(\text{mod 1}),  \label{hipcont1}
\end{equation}
we have:

\begin{equation}
p_{2}\circ \widetilde{f}^{n^{+}}(p_{1}(\widetilde{z}^{+}),t)<2+t\leq 3
\label{hipcont2}
\end{equation}

As $p_{2}\circ \widetilde{f}^{n^{+}}(\widetilde{z}^{+})-p_{2}(\widetilde{z}%
^{+})>8+4\left| k\right| M,$ there exists $\widetilde{z}^{*}\in r$ such that:

\[
\begin{array}{c}
p_{2}\circ \widetilde{f}^{n^{+}}(\widetilde{z}^{*})=4+2\left| k\right| M \\ 
\text{and} \\ 
\begin{array}{l}
\forall \text{ }\widetilde{m}\in r^{\prime }:=\{\text{linear segment of
extremes }\widetilde{z}^{*}\text{ and }\widetilde{z}^{+}\}, \\ 
p_{2}\circ \widetilde{f}_{k,g}^{n^{+}}(\widetilde{m})\geq 4+2\left| k\right|
M
\end{array}
\end{array}
\]
The reason why such a point $\widetilde{z}^{*}$ exists is the following: For
at least one integer $s,$ $p_{1}(\widetilde{z}^{+})+s\in p_{1}\left( 
\widetilde{f}_{k,g}^{n^{+}}(r)\right) .$ Thus, from (\ref{hipcont1}) and (%
\ref{hipcont2}), $\widetilde{f}_{k,g}^{n^{+}}(r)$ must cross the line $l$
given by: $l:=\mathbb{R}\times \{4+2\left| k\right| M\}$

Also from (\ref{hipcont1}) and (\ref{hipcont2}) we have: 
\[
{\sup_{\widetilde{m},\widetilde{k}\in {r}^{\prime }}}\left| p_{1}\circ 
\widetilde{f}_{k,g}^{n^{+}}(\widetilde{m})-p_{1}\circ \widetilde{f}%
_{k,g}^{n^{+}}(\widetilde{k})\right| <1
\]

It is clear that: 
\[
p_{2}\circ \widetilde{f}_{k,g}^{n^{+}}(\widetilde{z}^{+})-p_{2}\circ 
\widetilde{f}_{k,g}^{n^{+}}(\widetilde{z}^{*})>4+2\left| k\right| M
\]

Let us consider $\widetilde{f}_{k,g}(\widetilde{f}_{k,g}^{n^{+}}(r^{\prime})):$%
\[
{\sup_{\widetilde{m},\widetilde{k}\in {r}^{\prime }}}\left| p_{1}\circ 
\widetilde{f}_{k,g}^{1+n^{+}}(\widetilde{m})-p_{1}\circ \widetilde{f}%
_{k,g}^{1+n^{+}}(\widetilde{k})\right| \geq 4+2\left| k\right| M-1-2\left|
k\right| M>2
\]

And finally, from the choice of $r^{\prime },$ the following holds 

\[
\inf_{\widetilde{m}\in r^{\prime }}p_{2}\left(
\widetilde{f}_{k,g}^{1+n^{+}}(\widetilde{m})\right)
\geq 4+2|k|M-|k|M\geq4.
\]
So, either the lemma is proved for $n_{1}=1+n^{+},$ or our initial assumption
that for all $\widetilde{z}\in [0,1]^{2}$ such that $\widetilde{f}%
_{k,g}^{n^{+}}(\widetilde{z})=\widetilde{z}+(s,m)$ for some integer $s,$
implies $m<2$ does not hold, and in this case the lemma is proved for $n_{1}=n^{+}.$
\end{proof}

\begin{lemma}
\label{noinvcurves} Assume $g,g^{*}:$ $\mathbb{R}\rightarrow \mathbb{R}$
are 1-periodic continuous functions with zero mean. If $\widehat{f}_{k,g}$
does not have RICs, then there exists $\varepsilon >0$ such that for 
\[
\left| k-k^{*}\right| +\left\| g-g^{*}\right\| _{\infty }<\varepsilon ,
\]
the map $\widehat{f}_{k^{*},g^{*}}$ also lacks RICs.
\end{lemma}

\begin{proof}
As, $\widehat{f}_{k,g}$ does not have RICs, for $M=8+4(\left| k\right|
+1).(\left\| g\right\| _{\infty }+1),$ from Lemma \ref{birkhoff}, there
exists $\widetilde{z}^{+}\in [0,1]^{2}$ and an integer $n^{+}>0$ such that $%
p_{2}\circ \widetilde{f}_{k,g}^{n^{+}}(\widetilde{z}^{+})-p_{2}(\widetilde{z}%
^{+})>M.$ By continuity, if $0<\varepsilon <1$ is sufficiently small, then
the same holds for $\widetilde{f}_{k^{*},g^{*}}.$ 

By contradiction, assume that 
$\widehat{f}_{k^{*},g^{*}}$ has an invariant curve $\gamma$. So, $\gamma+(0,i)$ is also invariant 
for all integers $i$. This means that if the positive orbit of a point returns to the vertical containing that 
point, the total vertical displacement must stay between -1 and 1. So, from the %using the
%positive-displacement argument in the proof of 
positive-displacement argument in the proof of Lemma~\ref{lemBD01} we
get a contradiction, %from the proof of Theorem 3 of \cite{non1} that 
and thus $\widehat{f}_{k^{*},g^{*}}$ does not
have RICs.  
\end{proof}

\subsection{ Bullett's example}

Consider the $1$-periodic, zero mean, continuous function $g_{0}:$ $\mathbb %
R\rightarrow \mathbb R,$ given by: 
\begin{equation}
\text{for }0\leq \widetilde{x}\leq 1,\text{ }g_{0}(\widetilde{x})=\left| 
\widetilde{x}-\frac{1}{2}\right| -\frac{1}{4}=\begin{cases}
\frac14-\widetilde{x},&0\le \widetilde{x}\le\frac12,\\[1mm]
\widetilde{x}-\frac34,&\frac12\le \widetilde{x}\le1, \end{cases} .
\label{g0}
\end{equation}
In \cite{Bullett}, S. Bullett proved that $\widehat{f}_{k,g_{0}}$ has no
RIC for $0.918<k<1$ and also for $k>4/3.$ Exactly at 
$4/3,$ there exists a RIC.

In the following, we analyze the situation for $k=4/3$ in detail.% the map $\widehat{f}_{4/3,g_{0}}$ and its
%invariant curve.

Define the $1$-periodic function 
\begin{equation}
u_{0}(\widetilde{x})=\begin{cases} \frac23\widetilde{x}-\frac12,&0\le
\widetilde{x}\le\frac34,\\[1mm] \frac32-2\widetilde{x},&\frac34\le
\widetilde{x}\le1. \end{cases}  \label{eq:u0}
\end{equation}
Its endpoint values agree, so its periodic extension is continuous. The
circle map induced on $\widehat{\Gamma }_{u_{0}}=\pi (graph(u_{0}))$ is
denoted by $h_{0}$. Consider its real lift given by 
\begin{equation}
\begin{aligned}
\widetilde h_0(\widetilde x)
&:=\widetilde x+u_0(\widetilde x)+\frac43g_0(\widetilde x)
=\begin{cases}
\widetilde x/3-1/6,&0\le\widetilde x\le1/2,\\[1mm]
3\widetilde x-3/2,&1/2\le\widetilde x\le3/4,\\[1mm]
\widetilde x/3+1/2,&3/4\le\widetilde x\le1,
\end{cases} \text{ and }\widetilde h_0(\widetilde x+1)=\widetilde h_0(\widetilde x)+1.%\\[1mm]\\[1mm]
%\widetilde h_0(\widetilde x+1)&=\widetilde h_0(\widetilde x)+1.
\end{aligned}
\label{eq:h0}
\end{equation}
This is a strictly increasing homeomorphism of $\mathbb R$. Inverting the
branches, with their degree-one extensions, gives 
\begin{equation}
\begin{aligned}
\widetilde h_0^{-1}(\widetilde x)
&=\begin{cases}
\widetilde x/3+1/2,&0\le\widetilde x\le3/4,\\[1mm]
3\widetilde x-3/2,&3/4\le\widetilde x\le1, 
\end{cases} \text{ and }\widetilde h_0^{-1}(\widetilde x+1)=\widetilde h_0^{-1}(\widetilde x)+1.%\\[1mm]
%\widetilde h_0^{-1}(\widetilde x+1)&=\widetilde h_0^{-1}(\widetilde x)+1.
\end{aligned}
\label{eq:h0inverse}
\end{equation}
In particular, 
\begin{equation}
\widetilde{h}_{0}^{-1}(\widetilde{x})=\widetilde{x}-u_{0}(\widetilde{x}%
),\qquad \widetilde{h}_{0}(\widetilde{x})+\widetilde{h}_{0}^{-1}(\widetilde{x%
})-2\widetilde{x}=\frac43g_{0}(\widetilde{x}).  \label{eq:seed-identities}
\end{equation}
These identities immediately verify invariance: 
\[
u_{0}(\widetilde{h}_{0}(\widetilde{x}))=\widetilde{h}_{0}(\widetilde{x})-%
\widetilde{x}=u_{0}(\widetilde{x})+\frac43g_{0}(\widetilde{x}),
\]
and consequently 
\begin{equation}
\widetilde{f}_{4/3,g_{0}}(\widetilde{x},u_{0}(\widetilde{x}))=(\widetilde{h}%
_{0}(\widetilde{x}),u_{0}(\widetilde{h}_{0}(\widetilde{x}))).
\label{eq:seed-invariance}
\end{equation}
Projecting by $\pi $ proves that $\widehat{\Gamma }_{u_{0}}$ is invariant
under $\widehat{f}_{4/3,g_{0}}$.

The displacement $p_{0}=\widetilde{h}_{0}-\operatorname{id}$ is particularly
useful: 
\begin{equation}
p_{0}(\widetilde{x})=\begin{cases} -\frac23\widetilde{x}-\frac16,&0\le
\widetilde{x}\le\frac12,\\[1mm] 2\widetilde{x}-\frac32,&\frac12\le
\widetilde{x}\le\frac34,\\[1mm] -\frac23\widetilde{x}+\frac12,&\frac34\le
\widetilde{x}\le1. \end{cases}  \label{eq:p0}
\end{equation}
Thus 

\begin{equation}
p_{0}(\widetilde{x})\le 0,\qquad \max_{\widetilde{x}\in \mathbb{R}}p_{0}(\widetilde{x}%
)=0,\qquad p_{0}(\widetilde{x})=0\ \Longleftrightarrow \ \widetilde{x}\in 
\frac{3}{4}+\mathbb Z.  \label{eq:maximum}
\end{equation}
The lift $\widetilde{h}_{0}$ has rotation number $0$. Its unique fixed point
modulo $1$ repels from the left and attracts from the right. The RIC 
$\widehat{\Gamma }_{u_{0}}$ comes from a saddle connection between one
stable and one unstable branch of the hyperbolic $\widehat{f}_{4/3,g_{0}}$%
-fixed point $\widehat{P}=(3/4,0).$ A simple computation shows that its
eigenvalues are $3$ and $1/3.$ In this way, $\widehat{\Gamma }_{u_{0}}$ has
a corner at $\widehat{P}.$

Both $\widetilde{h}_0$ and $\widetilde{h}_0^{-1}$ are $3$-Lipschitz. In
particular, 
\begin{equation}  \label{eq:inverse-Lip}
|\widetilde{h}_0^{-1}(s)-\widetilde{h}_0^{-1}(t)|\le3|s-t|.
\end{equation}

\subsection{Ordinary Fourier truncation and the Gibbs estimate}

For a 1-periodic integrable scalar function $v$, use the convention 
\[
c_{j}(v)=\int_{0}^{1}v(\widetilde{x})e^{-2\pi ij\widetilde{x}}\,d\widetilde{x%
},\qquad S_{n}v(\widetilde{x})=\sum_{j=-n}^{n}c_{j}(v)e^{2\pi ij\widetilde{x}%
}.
\]
Define 
\begin{equation}
p_{n}=S_{n}p_{0},\qquad \widetilde{h}_{n}(\widetilde{x})=\widetilde{x}+p_{n}(%
\widetilde{x}),  \label{eq:hn}
\end{equation}
where $p_{0}$ comes from expression (\ref{eq:p0}). Here $p_{n}$ is a real
trigonometric polynomial. The inequality to establish is $p_{n}^{\prime }>-1$%
, equivalently $\widetilde{h}_{n}^{\prime }>0$.

\subsubsection{An explicit Fourier formula and a uniform error bound}

From \eqref{eq:p0}, 
\[
p_{0}^{\prime }(\widetilde{x})=-\frac{2}{3}+\frac{8}{3}\,\chi _{(1/2,\,3/4)}(%
\widetilde{x})\text{ for almost every }\widetilde{x},\text{ and }%
\int_{0}^{1}p_{0}(\widetilde{x})\,d\widetilde{x}=-\frac{1}{4}. 
\]
For $j\ne 0$, integration by parts therefore gives 
\[
c_{j}(p_{0})=-\frac{2}{3\pi ^{2}j^{2}}\left( e^{-2\pi ij(1/2)}-e^{-2\pi
ij(3/4)}\right) . 
\]
Equivalently, 
\begin{equation}
p_{n}(\widetilde{x})=-\frac{1}{4}-\frac{4}{3\pi ^{2}}\sum_{j=1}^{n}\frac{%
\cos (2\pi j(\widetilde{x}-\frac{1}{2}))-\cos (2\pi j(\widetilde{x}-\frac{3}{%
4}))}{j^{2}}.  \label{eq:fourier-explicit}
\end{equation}
The infinite series is absolutely and uniformly convergent to $p_{0}$. In
particular, for $n\ge 1$, 
\begin{equation}
\begin{aligned}
e_n&:=\|\widetilde h_n-\widetilde h_0\|_\infty
=\|p_n-p_0\|_\infty\\
&\le\frac{8}{3\pi^2}\sum_{j>n}\frac1{j^2}
\le\frac{8}{3\pi^2n}.
\end{aligned}
\label{eq:en}
\end{equation}

\subsubsection{Positive derivatives for all sufficiently large orders}

We use the classical Gibbs estimate for the characteristic function $\chi
_{I}(x)$ of a fixed nondegenerate proper interval $I\subset \mathbb T$: 
\begin{equation}
\min_{\widetilde{x}}S_{n}\chi _{I}(\widetilde{x})=-G+o(1),\text{ where }G=\frac{1}{\pi }%
\int_{0}^{\pi }\frac{\sin t}{t}\,dt-\frac{1}{2}\cong 0.089489872\ldots .
\label{eq:gibbs}
\end{equation}
The convention here measures the undershoot as a fraction of the jump, which
is $1$. The sine-integral description of this phenomenon is recalled in 
\cite[\S6.16(i)]{DLMF}.

The local profile at a rising endpoint is $1/2+\Si(t)/\pi$, where
$\Si(t)=\int_0^t(\sin s)/s\,ds$. It has minimum $-G$; the falling endpoint gives its reflection.
Uniform estimates for this profile near each endpoint, together with uniform
convergence away from the endpoints, give \eqref{eq:gibbs}.

\begin{proposition}
\label{prop:positive} For all sufficiently large $n$, $\widetilde{h}_{n}$ is
the lift of a real-analytic orientation-preserving circle diffeomorphism.
More precisely, 
\[
\lim_{n\to \infty }\min_{\widetilde{x}\in \mathbb{R}}\widetilde{h}_{n}^{\prime }(%
\widetilde{x})=\frac{1}{3}-\frac{8}{3}G=0.094693674\ldots>\frac1{135}>0.
\]
\end{proposition}

\begin{proof}
Since $p_{0}$ is absolutely continuous, differentiation of its finite
Fourier sums gives 
\[
\widetilde{h}_{n}^{\prime }=1+(S_{n}p_{0})^{\prime }=1+S_{n}(p_{0}^{\prime
})=\frac{1}{3}+\frac{8}{3}S_{n}\chi _{(1/2,\,3/4)}.
\]
By \eqref{eq:gibbs}, therefore,
\[
\min_{\widetilde{x}\in \mathbb{R}}\widetilde h_n^{\prime}(\widetilde x)
=\frac13-\frac83G+o(1)\qquad(n\to\infty).
\]
Let us find a short rational upper bound for $G:$ Integrating the Taylor
series of $\frac{\sin t}{t}$ term by term on $[0,\pi ]$, which is justified
by uniform convergence, gives 
\[
G=-\frac{1}{2}+\sum_{m=0}^{\infty }(-1)^{m}a_{m},\qquad a_{m}=\frac{\pi ^{2m}%
}{(2m+1)(2m+1)!}.
\]
%See also \cite[\S6.6, equation~(6.6.5)]{DLMF} for the sine-integral series.
The terms are strictly decreasing, since 
\[
\frac{a_{m+1}}{a_{m}}=\frac{\pi ^{2}(2m+1)}{(2m+3)^{2}(2m+2)}\le \frac{\pi
^{2}}{18}<1,\qquad m\ge 0,
\]
and they tend to zero. To verify the displayed inequality, note that 
\[
(2m+3)^{2}(2m+2)-18(2m+1)=8m^{3}+32m^{2}+6m\ge 0\text{ for all integers }m>0.
\]
Consequently, writing 
\[
B_{N}=-\frac{1}{2}+\sum_{m=0}^{N}(-1)^{m}a_{m},
\]
the alternating-series remainder estimate yields 
\begin{equation}
B_{2r+1}<G<B_{2r},\qquad |G-B_{N}|<a_{N+1}.
\end{equation}

For the derivative estimate, it suffices to consider the first three terms of the series, and
use that $49/5<\pi ^{2}<10$: 
\[
G<B_{2}=\frac{1}{2}-\frac{\pi ^{2}}{18}+\frac{\pi ^{4}}{600}<\frac{11}{90}. 
\]
Thus positivity in Proposition~\ref{prop:positive} does not depend on any
uncertified decimal computation. In fact, its limiting minimum exceeds $%
1/135:$%
\[
\frac{1}{3}-\frac{8}{3}G>\frac{1}{3}-\frac{8}{3}\frac{11}{90}=\frac{1}{135}%
>0. 
\]
Thus $\widetilde{h}_{n}^{\prime }$ is bounded below by a positive constant
for all sufficiently large $n$. Since $\widetilde{h}_{n}(\widetilde{x}+1)=%
\widetilde{h}_{n}(\widetilde{x})+1$, it is an increasing bijection of $%
\mathbb R$. The analytic inverse function theorem gives a real-analytic
inverse. \end{proof}

The role of the Gibbs estimate is quantitative: uniform convergence of $%
\widetilde{h}_{n}$ to $\widetilde{h}_{0}$ alone would not establish
positivity of the derivatives.

\subsection{Herman's idea}

In this subsection, we present a clever idea from Michael Herman, see Chapter~II of \cite{aste}, Remark~2.3.3, p.~53. Roughly speaking, from a circle
diffeomorphism $h$, he builds a RIC $\gamma $ for a
particular standard map, for which the dynamics of the standard map
restricted to $\gamma $ is given by $h.$ We will state it in a form adequate
for our application.

\begin{lemma}\label{her}
Let $\widetilde{h}:\mathbb R\to \mathbb R$ be the lift to $\mathbb R$ of an
orientation preserving real analytic circle diffeomorphism $h.$ In
particular, $\widetilde{h}$ is also a real-analytic diffeomorphism with $%
\widetilde{h}^{\prime }>0$ and $\widetilde{h}(\widetilde{x}+1)=\widetilde{h}(%
\widetilde{x})+1$. For any $k^{*}>0$, put 
\begin{equation}
g_{\widetilde{h}}(\widetilde{x})=\frac{\widetilde{h}(\widetilde{x})+%
\widetilde{h}^{-1}(\widetilde{x})-2\widetilde{x}}{k^{*}}\text{ and }u_{%
\widetilde{h}}(\widetilde{x})=\widetilde{x}-\widetilde{h}^{-1}(\widetilde{x}%
).
\end{equation}
Then $g_{\widetilde{h}},u_{\widetilde{h}}$ are real analytic and $1$%
-periodic, $\int_{0}^{1}g_{\widetilde{h}}(\widetilde{x})\,d\widetilde{x}=0$,
and 
\[
\widetilde{f}_{k^{*},g_{\widetilde{h}}}(\widetilde{x},u_{\widetilde{h}}(%
\widetilde{x}))=(\widetilde{h}(\widetilde{x}),u_{\widetilde{h}}(\widetilde{h}%
(\widetilde{x}))).
\]
\end{lemma}

\begin{proof}
Analyticity and periodicity follow from the corresponding properties of $%
\widetilde{h}$ and $\widetilde{h}^{-1}$. Furthermore, 
\[
u_{\widetilde{h}}(\widetilde{x})+k^{*}g_{\widetilde{h}}(\widetilde{x})=%
\widetilde{h}(\widetilde{x})-\widetilde{x}=u_{\widetilde{h}}(\widetilde{h}(%
\widetilde{x})), 
\]
which proves invariance and identifies the induced lift as $\widetilde{h}$.

To check the mean, let $p=\widetilde{h}-\operatorname{id}$. Since $%
\widetilde{h}^{-1}(\widetilde{y})-\widetilde{y}$ is periodic, a change of
variables $\widetilde{y}=\widetilde{h}(\widetilde{x})$ gives 
\begin{align*}
\int_{0}^{1}(\widetilde{h}^{-1}(\widetilde{y})-\widetilde{y})\,d\widetilde{y}%
& =\int_{\widetilde{h}(0)}^{\widetilde{h}(0)+1}(\widetilde{h}^{-1}(%
\widetilde{y})-\widetilde{y})\,d\widetilde{y} \\
& =-\int_{0}^{1}p(\widetilde{x})\widetilde{h}^{\prime }(\widetilde{x})\,d%
\widetilde{x} \\
& =-\int_{0}^{1}p(\widetilde{x})\,d\widetilde{x}-\int_{0}^{1}p(\widetilde{x}%
)p^{\prime }(\widetilde{x})\,d\widetilde{x} \\
& =-\int_{0}^{1}p(\widetilde{x})\,d\widetilde{x},
\end{align*}
because $\int_{0}^{1}p(\widetilde{x})p^{\prime }(\widetilde{x})\,d\widetilde{%
x}=0.$ As $\widetilde{h}(\widetilde{x})-\widetilde{x}=p(\widetilde{x})$ and 
\[
\int_{0}^{1}g_{\widetilde{h}}(\widetilde{x})\,d\widetilde{x}=\frac{1}{k^{*}}%
\left[ \int_{0}^{1}(\widetilde{h}^{-1}(\widetilde{y})-\widetilde{y})\,d%
\widetilde{y}+\int_{0}^{1}(\widetilde{h}(\widetilde{x})-\widetilde{x})\,d%
\widetilde{x}\right] , 
\]
we get that $\int_{0}^{1}g_{\widetilde{h}}(\widetilde{x})\,d\widetilde{x}=0.$
\end{proof}

\section{Proof of Theorem~\ref{thm:main}}\label{sec:proof}

\hskip0.6truecm \textit{Part 1. No RICs for Bullett's example at k=0.95;}
\vskip0.1truecm

From Bullett's paper \cite{Bullett}, the closed set $A_{g_{0}}$ associated
with the function $g_0$ in \eqref{g0} avoids $]0.918,1[.$ %From expression (\ref{sessao}), $%
%\rho _{V}(\widehat{f}_{k,g_{0}})=[\rho _{V}^{-}(k,g_{0}),\rho
%_{V}^{+}(k,g_{0})]$ satisfies 
%\begin{equation}
%\rho _{V}^{-}(k,g_{0})<0<\rho _{V}^{+}(k,g_{0})\text{ for all }k\in
%[0.92,0.99].  \label{semcri}
%\end{equation}
\vskip0.2truecm
\textit{Part 2. }$C^{0}$\textit{-approximation of }$h_{0},\widetilde{h}_{0}$%
\textit{\ by a %real analytic 
pair }$\sigma _{n},\widetilde{\sigma }_{n},$\textit{
whose rotation number is Diophantine;}
\vskip0.1truecm
From Proposition~\ref{prop:positive}, consider $n_0>0$ sufficiently large such that for 
$n\geq n_0$, $\widetilde{h}_n^{\prime}(\widetilde{x})>0$ for
all $\widetilde{x}\in \mathbb{R}$. Define 
\[
M_{n}=\max_{\widetilde{x}\in [0,1]}p_{n}(\widetilde{x}). 
\]
Equations \eqref{eq:maximum} and \eqref{eq:en} imply 
\begin{equation}
|M_{n}|\le e_{n}\text{ and }p_{n}(\widetilde{x})\le e_{n}\quad \text{for all 
}\widetilde{x}\in \mathbb R.  \label{eq:Mn}
\end{equation}
Indeed $p_{n}\le p_{0}+e_{n}\le e_{n}$, whereas $p_{n}(3/4)\ge
p_{0}(3/4)-e_{n}=-e_{n}$.

Consider the family of translated lifts 
\begin{equation}
\widetilde{h}_{n,t}(\widetilde{x})=\widetilde{h}_{n}(\widetilde{x})-M_{n}-%
\frac{t}{n},\qquad 0\le t\le 1.  \label{eq:translated-family}
\end{equation}
All these maps are real-analytic diffeomorphisms, since their derivatives
are $\widetilde{h}_{n}^{\prime }>0$.

At $t=0$, the function $\widetilde{h}_{n,0}(\widetilde{x})-\widetilde{x}%
=p_{n}(\widetilde{x})-M_{n}$ has zero as a maximum, so at the $\widetilde{x}$
where this maximum is attained, $\widetilde{h}_{n,0}$ has a fixed point,
which implies that 
\begin{equation}
\rho (\widetilde{h}_{n,0})=0.  \label{eq:rot-left}
\end{equation}
For $t>0$, 
\[
\widetilde{h}_{n,t}(\widetilde{x})\le \widetilde{x}-\frac{t}{n}. 
\]
Applying this inequality repeatedly yields 
\[
\widetilde{h}_{n,t}^{\,m}(\widetilde{x})\le \widetilde{x}-\frac{mt}{n}, 
\]
and hence 
\begin{equation}
\rho (\widetilde{h}_{n,t})\le -\frac{t}{n}<0.  \label{eq:rot-right}
\end{equation}

As the rotation number is continuous in the uniform topology on lifts of orientation-preserving
homeomorphisms of the circle, $[0,1]\ni t\mapsto \rho (\widetilde{h}_{n,t})$
is continuous, and its image contains $[-1/n,0]$. Choose a Diophantine
number, for instance, 
\[
\alpha _{n}:=-\sqrt{2}/(2n)\in (-1/n,0). 
\]
By the intermediate value theorem, there exists $t_{n}\in (0,1)$ such that 
\begin{equation}
\widetilde{\sigma }_{n}:=\widetilde{h}_{n,t_{n}}\text{ satisfies }\rho (%
\widetilde{\sigma }_{n})=\alpha _{n}.  \label{eq:chosen-H}
\end{equation}
Crucially, as $M_{n}+1/n\rightarrow 0$ as $n\rightarrow \infty ,$ we get
that: 
\begin{equation}
\begin{aligned}
\|\widetilde\sigma_n-\widetilde h_0\|_\infty
&=\|\widetilde h_{n,t_n}-\widetilde h_0\|_\infty\\
&\le e_n+|M_n|+\frac1n
\le2e_n+\frac1n
\le\frac{1+16/(3\pi^2)}n.
\end{aligned}
\label{eq:H-close}
\end{equation}
Thus selection of the Diophantine rotation number is compatible with
remaining arbitrarily close to the original piecewise-linear lift $%
\widetilde{h}_{0}.$

\vskip0.2truecm

\textit{Part 3. Estimate of }$\left\| \theta _{n}-g_{0}\right\| _{\infty }$ 
\textit{and non-existence of RICs;}
\vskip0.1truecm

First, apply Herman's construction (Lemma \ref{her}) with $k^*=4/3,$ $\widetilde{h}=\widetilde\sigma_n$, and set
\begin{equation}
\theta_n(\widetilde x)=\frac{\widetilde\sigma_n(\widetilde x)
       +\widetilde\sigma_n^{-1}(\widetilde x)-2\widetilde x}{4/3} \text{ and } 
u_n(\widetilde x)=\widetilde x-\widetilde\sigma_n^{-1}(\widetilde x).
\label{eq:theta-n}
\end{equation}

Since
\[
g_0(\widetilde x)
=\frac{\widetilde h_0(\widetilde x)
       +\widetilde h_0^{-1}(\widetilde x)-2\widetilde x}{4/3},
\]
we obtain that
\[
\|\theta_n-g_0\|_\infty
\le\frac34\left(
\|\widetilde\sigma_n-\widetilde h_0\|_\infty
+\|\widetilde\sigma_n^{-1}-\widetilde h_0^{-1}\|_\infty
\right).
\]
For any $\widetilde y\in\mathbb R$, choose
$\widetilde x=\widetilde\sigma_n^{-1}(\widetilde y)$.
The $3$-Lipschitz estimate \eqref{eq:inverse-Lip} gives
\begin{align*}
\bigl|\widetilde\sigma_n^{-1}(\widetilde y)
      -\widetilde h_0^{-1}(\widetilde y)\bigr|
&=\bigl|\widetilde x
      -\widetilde h_0^{-1}(\widetilde\sigma_n(\widetilde x))\bigr|\\
&=\bigl|\widetilde h_0^{-1}(\widetilde h_0(\widetilde x))
      -\widetilde h_0^{-1}(\widetilde\sigma_n(\widetilde x))\bigr|\\
&\le3\bigl|\widetilde h_0(\widetilde x)-\widetilde\sigma_n(\widetilde x)\bigr|\\
&\le3\|\widetilde\sigma_n-\widetilde h_0\|_\infty.
\end{align*}
Taking the supremum over $\widetilde y$ yields
\[
\|\widetilde\sigma_n^{-1}-\widetilde h_0^{-1}\|_\infty
\le3\|\widetilde\sigma_n-\widetilde h_0\|_\infty.
\]
Consequently, by \eqref{eq:H-close},
\begin{equation}
\begin{aligned}
\|\theta_n-g_0\|_\infty
&\le\frac34 4\|\widetilde\sigma_n-\widetilde h_0\|_\infty\\
&=3\|\widetilde\sigma_n-\widetilde h_0\|_\infty
\le\frac{3+16/\pi^2}{n}\longrightarrow0.
\end{aligned}
\label{eq:theta-close}
\end{equation}
So, as $\widehat{f}_{0.95,g_0}$ has no RICs, Lemma \ref{noinvcurves} implies that there exists $n_{1}>n_{0}>0$ such
that for all $n\geq n_{1},$

\begin{equation}
\widehat{f}_{0.95,\theta_n} \text{ does not have RICs either.}
\end{equation}

\vskip0.2truecm
\textit{Part 4. Persistence of real-analytic RICs near $0$ and near $4/3$.}
\vskip0.1truecm
Fix a sufficiently large $n>n_1$, and write
\[
\widetilde{\sigma}=\widetilde{\sigma}_n,\qquad
\theta=\theta_n,\qquad u=u_n,\qquad \alpha=\alpha_n.
\]
By the construction above, $\widehat{f}_{4/3,\theta}$ has the
real-analytic RIC $\widehat{\Gamma}_u=\pi(\operatorname{graph}(u))$,
and its dynamics on this curve has the Diophantine rotation number
$\alpha$. From now on, $\theta$ remains fixed while $k$ varies.

We use the following standard local persistence form of the KAM theorem
for exact symplectic twist maps: 
\vskip0.1truecm
{\bf KAM Theorem.} \emph{If a real-analytic exact symplectic
strict twist map has a real-analytic RIC with Diophantine rotation number
$\alpha$, then every sufficiently small analytic exact symplectic
perturbation has a nearby real-analytic RIC with the same rotation number
$\alpha$.} 
\vskip0.2truecm
Here smallness is measured on a fixed complex neighborhood
of the original curve. This is the strict twist-map consequence of the
local KAM theorem (see, for example, \cite[Theorem~2.5]{FHL}), using
analytic linearization of the dynamics on the curve \cite{Yoccoz}.
The strict twist condition supplies the required nondegeneracy.

All these hypotheses hold here. The maps $\widehat{f}_{k,\theta}$ are
real analytic and exact symplectic, since $\theta$ has zero mean.
Moreover, their lifts satisfy the strict twist condition
\[
\frac{\partial\bigl(p_1\circ\widetilde{f}_{k,\theta}\bigr)}
{\partial\widetilde{y}}=1.
\]
Since $\theta$ is fixed and real analytic, we can choose a fixed complex
neighborhood $U$ of the lifted curve, modulo horizontal integer
translations, on which $\theta$ extends holomorphically and is bounded.
On this neighborhood,
\[
\widetilde{f}_{k,\theta}(\widetilde{x},\widetilde{y})
-\widetilde{f}_{4/3,\theta}(\widetilde{x},\widetilde{y})
=(k-4/3)\theta(\widetilde{x})\binom{1}{1}.
\]
Consequently, in the supremum norm on $U$, there is a constant $C_U>0$
independent of $k$ such that
\[
\bigl\|\widetilde{f}_{k,\theta}
-\widetilde{f}_{4/3,\theta}\bigr\|_U
\le C_U|k-4/3|\longrightarrow0
\qquad\text{as }k\longrightarrow4/3.
\]
The persistence theorem therefore gives $\delta_{4/3}>0$ such that, for every
$|k-4/3|<\delta_{4/3}$, the map $\widehat{f}_{k,\theta}$ has a real-analytic
RIC with rotation number $\alpha$. Thus
\begin{equation}
(4/3-\delta_{4/3},4/3+\delta_{4/3})\subset A_\theta.
\label{eq:persistent-interval}
\end{equation}
%In particular, the closed interval
%$[4/3-\delta/2,4/3+\delta/2]$ is contained in $A_\theta$ and has $4/3$
%in its interior.

To obtain an interval around $0$, observe that
\[
\widetilde{f}_{0,\theta}(\widetilde{x},\widetilde{y})
=(\widetilde{x}+\widetilde{y},\widetilde{y}).
\]
Thus the horizontal circle
$\pi(\mathbb R\times\{\alpha\})$ is a real-analytic RIC with the same
Diophantine rotation number $\alpha$. The integrable map
$\widehat{f}_{0,\theta}$ is exact symplectic and satisfies the strict
twist condition. Moreover,
\[
\widetilde{f}_{k,\theta}(\widetilde{x},\widetilde{y})
-\widetilde{f}_{0,\theta}(\widetilde{x},\widetilde{y})
=k\theta(\widetilde{x})\binom{1}{1},
\]
so these maps converge to $\widehat{f}_{0,\theta}$ in the analytic norm
on a fixed complex neighborhood of this horizontal circle as $k\to0$.
The same persistence theorem gives $\delta_0>0$ such that every
$\widehat{f}_{k,\theta}$ with $|k|<\delta_0$ has a real-analytic RIC
with rotation number $\alpha$. Hence
\begin{equation}
(-\delta_0,\delta_0)\subset A_\theta.
\label{eq:persistent-zero-interval}
\end{equation}
%The latter is a closed interval containing $0$ in its interior.

Finally, let $I_0$ and $I_{4/3}$ be the connected components of $A_\theta$
containing $0$ and $4/3$, respectively. Part~3 implies that
$\theta\not\equiv0$, since $\widehat{f}_{k,0}$ has RICs for every $k$.
By item~\ref{item:parameter-set}, $A_\theta$ is compact.
Thus these components are compact intervals, and the neighborhoods just obtained
show that $0\in\operatorname{int}(I_0)$ and
$4/3\in\operatorname{int}(I_{4/3})$.
They are distinct: by Part~3, $0.95\notin A_\theta$.
%whereas a connected subset of $\mathbb R$ containing both $0$ and $4/3$
%would contain this entire gap. 
This proves Theorem~\ref{thm:main}. $\qed$

\begin{remark}
Although $\widetilde{\sigma }_{n}-\operatorname{id}$ is a trigonometric
polynomial plus a constant, $\theta _{n}$ also involves $\widetilde{\sigma }%
_{n}^{-1}$. The conclusion is real analyticity of $\theta _{n}$; no claim
that $\theta _{n}$ is a trigonometric polynomial is needed.
\end{remark}

\end{document}